\documentclass{amsart}
\usepackage[T1,T2A]{fontenc}	
\usepackage[utf8]{inputenc}	

\usepackage[a4paper, margin=50pt]{geometry}

\usepackage{subcaption}

\usepackage{amsmath}
\usepackage{amssymb}
\usepackage{amsthm}
\usepackage{xparse}
\usepackage{enumitem}
\usepackage{mathtools} 
\usepackage{tikz-cd}
\usepackage{csquotes}

\usepackage{array}
\newcolumntype{C}{>{$}c<{$}}

\usepackage[normalem]{ulem}
\usepackage[style=numeric,sorting=nty,backend=biber,language=british]{biblatex}
\usepackage{bm}
\usepackage{dsfont}

\usepackage{graphicx}
\graphicspath{{images/}}

\usepackage{tabularx}

\usepackage{hyperref}
\usepackage{xcolor}
\usepackage{tcolorbox}
\hypersetup{%
	colorlinks=true,
	linkcolor=blue!66!black,
	citecolor=red!66!black,
	urlcolor=green!33!black
}

\theoremstyle{definition}
\newtheorem{definition}{Definition}[section]
\newtheorem{theorem}[definition]{Theorem}
\newtheorem{lemma}[definition]{Lemma}
\newtheorem{example}[definition]{Example}
\newtheorem{proposition}[definition]{Proposition}
\newtheorem{conjecture}{Conjecture}[section]
\newtheorem{question}{Question}[section]

\newtheorem*{remark}{Remark}
\newtheorem{corollary}[definition]{Corollary}

\renewcommand{\l}{\lambda}
\newcommand\G{\Gamma}

\renewcommand{\emptyset}{\varnothing}

\usepackage{scalerel}

\usepackage{tikz}
\usetikzlibrary{hobby}
\newcommand{\eye}{%
	\kern0.1ex%
	\begin{tikzpicture}[use Hobby shortcut]%
		\draw[thick,line cap=round] (0ex,0.5ex) .. (1ex,1ex) .. (2ex,0.5ex);%
		\draw[thick,line cap=round] (0ex,0.5ex) .. (1ex,0ex) .. (2ex,0.5ex);%
		\filldraw[thick,line cap=round] (1ex,0.5ex) circle (0.15ex);%
	\end{tikzpicture}%
	\kern0.1ex%
}
\usepackage{eso-pic}

\renewcommand\bar{\overline}

\renewcommand\hat{\widehat}
\renewcommand\tilde{\widetilde}

\usepackage{extarrows}

\newcommand\seq{\Longrightarrow}

\renewcommand{\nabla}{\triangledown}

\let\ophi\phi
\let\ovarphi\varphi
\renewcommand{\phi}{\ovarphi}
\renewcommand{\varphi}{\ophi}

\newcommand\N{\mathbb{N}}
\newcommand\Z{\mathbb{Z}}
\newcommand\Q{\mathbb{Q}}
\newcommand\R{\mathbb{R}}
\renewcommand\C{\mathbb{C}}

\renewcommand\sl{\mathfrak{sl}}
\newcommand\so{\mathfrak{so}}
\newcommand\gl{\mathfrak{gl}}

\newcommand\FourT{\operatorname{4T}}

\newcommand\foot[1]{{\renewcommand{\thefootnote}{\ensuremath{\,\dagger\,}}\footnote{#1}\addtocounter{footnote}{-1}}}

\usepackage{float}

\newcommand{\setp}[1]{\left\{#1\right\}}
\newcommand{\setpp}[2]{\left\{#1\mid#2\right\}}

\newcommand{\norm}[1]{\left\|#1\right\|}

\makeatletter

\def\env@sqcases{%
	\let\@ifnextchar\new@ifnextchar
	\left\lbrack
	\def\arraystretch{1.2}%
	\array{@{}l@{\quad}l@{}}%
}
\makeatother

\ExplSyntaxOn
\NewDocumentCommand{\colvec}{O{,} m}{
	\vector_main:nnnn{\\}{#1}{#2}
}
\seq_new:N \l__vector_arg_seq
\cs_new_protected:Npn \vector_main:nnnn #1 #2 #3{
	\seq_set_split:Nnn \l__vector_arg_seq{#2}{#3}
	\begin{pmatrix}
		\seq_use:Nnnn \l__vector_arg_seq{#1}{#1}{#1}
	\end{pmatrix}
}
\ExplSyntaxOff

\setlist[itemize]{leftmargin=2cm}

\usepackage{xstring}
\usepackage{ifthen}
\usepackage{etoolbox}
\usetikzlibrary{
	knots,
	hobby,
	calc,
	shapes.geometric,
	plotmarks,
	arrows.meta,
	tqft,
	fit,
	patterns,
	patterns.meta,
	decorations.markings,
	fadings,
	decorations.pathmorphing,
	fixedpointarithmetic,
    external
}
\makeatletter
\tikzset{
	>={Latex[round,width=2mm,length=2mm]},
	l/.style={
		line cap=round,
		line join=round
	},
	c/.style={
		circle,
		inner sep=0pt,
		outer sep=0pt,
		minimum size=#1,
	},
	c/.default=1cm,
	cf/.style={
		rectangle,
		inner sep=2pt,
		outer sep=0pt,
		minimum size=#1,
	},
	cf/.default=10pt,
	show direction/.style={
		postaction=decorate,
		decoration={
			markings,
			mark=at position #1 with {\arrow[opacity=1]{>>}}
		}
	},
	tcenter/.style={
		baseline={([yshift=#1]current bounding box.center)}
	},
	tcenter/.default=-0.5ex,
    	nb/.style={
		draw,
		fill=white,
		thick,
		inner sep=0pt,
		align=center,
		minimum size=#1
	},
	nb/.default=10pt,
	nbe/.style={
		draw,
		fill=white,
		thick,
		inner sep=0pt,
		text width=20pt,
		align=center,
		minimum width=#1,
		minimum height=11pt
	},
	nbe/.default=22pt,
	nbec/.style n args={2}{
		draw,
		fill=white,
		thick,
		inner sep=0pt,
		align=center,
		minimum width=#1,
		minimum height=#2
	},
	nbec/.default={22pt}{11pt},
	pics/arrow/.style={
		code={
			\draw[line width=0pt, black, l, {<[line width=0.8pt,length=2mm,width=1mm]}-] (-0.5ex,0) -- (0.5ex,0);
		}
	}
}

\pgfdeclaredecoration{width to zero}{initial}{
	\state{initial}[width=0pt, next state=line, persistent precomputation={%
		\pgfmathdivide{\pgflinewidth}{\pgfdecoratedpathlength}%
		\let\speed=\pgfmathresult%
		\pgfmathmultiply{\speed}{.5pt}%
		\let\decrement=\pgfmathresult%
		\let\curwidth=\pgflinewidth%
 	}]{}
	\state{line}[width=.5pt, persistent postcomputation={%
		\pgfmathsubtract{\curwidth}{\decrement}%
		\let\curwidth=\pgfmathresult%
	}]{%
		\pgfsetlinewidth{\curwidth}%
		\pgfsetarrows{-}%
		\pgfpathmoveto{\pgfpointorigin}%
		\pgfpathlineto{\pgfqpoint{.75pt}{0pt}}%
		\pgfusepath{stroke}%
	}
	\state{final}{%
		\pgfsetlinewidth{\pgflinewidth}%
		\pgfpathmoveto{\pgfpointorigin}%
		\pgfusepath{stroke}%
	}
}
\tikzset{%
	vanish/.style={%
		fixed point arithmetic,
		decoration={%
			snake,
			segment length=#1,
			amplitude=0.6mm
		},
		decorate,
		postaction={%
			draw,
			decorate,
			decoration={width to zero}
		}
	},
	vanish/.default=6mm
}

\makeatother

\usepackage{fp}
\def\chordR{1}
\def\chordRad{1cm}
\def\chordCircle{(0,0) circle (\chordRad)}

\def\chordPointDSet#1#2#3{\coordinate (#1) at ($(0,0)!\chordR!90-360*#3/#2:(1,0)$)}
\DeclareRobustCommand{\chord}[3]{%
    \tikzset{external/export next=false}%
	\begin{tikzpicture}[scale=#1,tcenter]
		\draw[thick, black, l] \chordCircle;
		\foreach \i/\j in {#3} {
			\chordPointDSet{a}{#2}{\i};
			\chordPointDSet{b}{#2}{\j};
			\ifthenelse{\i=\j}{
				\coordinate (c) at ($(0,0)!1.25!(a)$);
				\pgfmathsetmacro\chordRHalf{0.25*\chordR}
				\draw[thick, black, l] (c) circle (\chordRHalf);
			}{
				\pgfmathtruncatemacro{\chordaout}{180+90-360*\i/#2}
				\pgfmathtruncatemacro{\chordain}{180+90-360*\j/#2}
				\draw[ultra thick, black, l] (a) to[out=\chordaout,in=\chordain] (b);
			}
		}
		\foreach \i/\j in {#3} {
			\chordPointDSet{a}{#2}{\i};
			\chordPointDSet{b}{#2}{\j};
			\draw[black,c=4pt]
				(a) node[fill] {}
				(b) node[fill] {};
		}
	\end{tikzpicture}%
}
\DeclareRobustCommand{\chordp}[3]{%
    \tikzset{external/export next=false}%
	\begin{tikzpicture}[scale=#1,tcenter]
		\draw[thick, black, l, dotted] \chordCircle;
		\foreach \i/\j in {#3} {
			\chordPointDSet{a}{#2}{\i};
			\chordPointDSet{b}{#2}{\j};
			\ifthenelse{\i=\j}{
				\coordinate (c) at ($(0,0)!1.25!(a)$);
				\pgfmathsetmacro\chordRHalf{0.25*\chordR}
				\draw[thick, black, l] (c) circle (\chordRHalf);
			}{
				\pgfmathtruncatemacro{\chordaout}{180+90-360*\i/#2}
				\pgfmathtruncatemacro{\chordain}{180+90-360*\j/#2}
				\draw[ultra thick, black, l] (a) to[out=\chordaout,in=\chordain] (b);
			}
		}
		\foreach \i/\j in {#3} {
			\chordPointDSet{a}{#2}{\i};
			\chordPointDSet{b}{#2}{\j};
			\draw[black,c=4pt]
				(a) node[fill] {}
				(b) node[fill] {};

			\pgfmathsetmacro\afrom{90-360*\i/#2-270/#2}
			\pgfmathsetmacro\ato{90-360*\i/#2+270/#2}
			\draw[black,thick,l]
				($\chordR * cos(\afrom) *(1,0) + \chordR * sin(\afrom) *(0,1)$) arc (\afrom:\ato:\chordR);

			\pgfmathsetmacro\bfrom{90-360*\j/#2-270/#2}
			\pgfmathsetmacro\bto{90-360*\j/#2+270/#2}
			\draw[black,thick,l]
				($\chordR * cos(\bfrom) *(1,0) + \chordR * sin(\bfrom) *(0,1)$) arc (\bfrom:\bto:\chordR);
		}
	\end{tikzpicture}%
}
\newcommand{\tchord}[2]{\chord{0.4}{#1}{#2}}
\newcommand{\tchordp}[2]{\chordp{0.4}{#1}{#2}}

\DeclareRobustCommand{\cgraph}[3]{%
    \tikzset{external/export next=false}%
	\begin{tikzpicture}[scale=#1,tcenter]
        \def\chordR{0.5}
        \path (0,-0.8) (0,0.8);
        \edef\cgraphn{#2}
        \def\cgraphone{1}
        \ifx\cgraphn\cgraphone%
            \draw[black,c=3pt]
                (0,0) node[fill] {};
        \else%
    		\foreach \i in {1,...,#2} {
    			\chordPointDSet{a}{#2}{\i};
    			\draw[black,c=3pt]
    				(a) node[fill] {};
    		}
        \fi
		\foreach \i/\j in {#3} {
			\chordPointDSet{a}{#2}{\i};
			\chordPointDSet{b}{#2}{\j};
			\draw[very thick, black, l]
				(a) -- (b);
		}
	\end{tikzpicture}%
}
\newcommand{\tcgraph}[2]{\,\cgraph{0.6}{#1}{#2}}

\newcounter{tbridgecount}
\newcommand{\tbridge}{%
	\kern0.11cm%
    \stepcounter{tbridgecount}%
    \tikz[remember picture,tcenter] \coordinate (tbridge-\thetbridgecount) at (0,0);
    \AddToShipoutPictureFG*{%
        \tikzset{external/export next=false}%
        \begin{tikzpicture}[scale=0.4,tcenter,use Hobby shortcut,remember picture,overlay]
            \begin{scope}[shift={(tbridge-\thetbridgecount)}]
                \path (0,-1) (0,1);
                \draw[fill=white,draw=none]
                    (-0.4,-0.2) rectangle (0.4,0.2);
                \draw[black,thick,l]
                    (-0.3,0.2) .. (0,0.1) .. (0.3,0.2)
                    (-0.3,-0.2) .. (0,-0.1) .. (0.3,-0.2);
            \end{scope}
        \end{tikzpicture}%
    }%
	\kern0.11cm%
}
\newcommand{\tbridgeinner}[1]{%
	\kern0.11cm%
    \tikz[remember picture,tcenter] \coordinate (tbridge-#1) at (0,0);%
    \AddToShipoutPictureFG*{%
        \tikzset{external/export next=false}%
        \begin{tikzpicture}[scale=0.4,tcenter,use Hobby shortcut,remember picture,overlay]
            \begin{scope}[shift={(tbridge-#1)}]
                \path (0,-1) (0,1);
                \draw[fill=white,draw=none]
                    (-0.4,-0.2) rectangle (0.4,0.2);
                \draw[black,thick,l]
                    (-0.3,0.2) .. (0,0.1) .. (0.3,0.2)
                    (-0.3,-0.2) .. (0,-0.1) .. (0.3,-0.2);
            \end{scope}
        \end{tikzpicture}%
    }%
	\kern0.11cm%
}
\renewcommand{\leadsto}{%
    \mathop{%
    	\kern.5ex%
    	\begin{tikzpicture}
    		\draw[very thick,black,l,->,decorate,decoration={snake,amplitude=0.6mm,post length=1mm}]
    			(0,0) -- (1,0);
    	\end{tikzpicture}%
    	\kern.5ex%
    }%
    \limits%
}

\title{$\sl(2)$-weight system does not extend to a graph 4-invariant}
\author{Daniil Fomichev}
\address{D.F.: Saint-Petersburg State University, 199034, 7/9 Universitetskaya Emb., Saint-Petersburg, Russia}
\email{fomichev.d.s@yandex.ru}
\author{Maksim Karev}
\address{M.K.: Guangdong Technion-Israel Institute of Technology, 515603, 241 Daxue Lu, Shantou city, Guangdong province, P.R. China}
\email{maksim.karev@gtiit.edu.cn}
\author{Fedor Pavutnitskiy}
\address{F.P.: Mathematical Institute, University of Oxford, Andrew Wiles Building, Radcliffe Observatory Quarter (550), Woodstock Road, Oxford, OX2 6GG}
\email{fedor.pavutnitskiy@maths.ox.ac.uk}
\author{Sergey Usanov}
\address{S.U.: Faculty of Mathematics, HSE University, 6 Usacheva St., 119048 Moscow, Russia}
\email{sergei.usanov.acad@gmail.com}

\date{\today}
\begin{document}
\begin{abstract}
A long-standing question by S.~Lando asks whether the $\sl(2)$-weight system extends to a unique 4-invariant of graphs.
We show that, in full generality, the answer to this question is negative. However, for certain specializations of the weight system, extensions do exist. Explicit formulae for computing two such specializations of the weight system are already known. We construct recurrence relations for one additional such extension and discuss the last remaining specialization, which conjecturally admits an extension.
We also study the polynomial coefficients of the $\sl(2)$-weight system and resolve the questions concerning their extension.
\end{abstract}

\maketitle

\tableofcontents

\section{Introduction}
\subsection{Motivation and history of the problem}
Weight systems are functions on the set of chord diagrams obeying so-called $4T$-relations. The weight systems are principal objects in the theory of finite type invariants of knots (see~\cite{V90} and~\cite{K93}). Metrized Lie (super)algebras can be used as a natural source of weight systems (see, for example,~\cite{K93, BN95, CDBook}). S.~Chmutov and S.~Lando have proved in~\cite{ChL07} that the values of weight systems associated with the Lie algebra $\sl(2)$ and the Lie superalgebra $\gl(1|1)$ are determined by the intersection graphs of chord diagrams, i.e. if two chord diagrams share the intersection graph, the values of these two weight systems on these two chord diagrams are the same. The graph-theoretic counterparts of weight systems are called 4-invariants (see~\cite{landohopf}), and every 4-invariant corresponds to a weight system through the intersection graph map. The map from the space of 4-invariants to the set of weight systems is non‑surjective. For instance, it is not even clear whether a weight system whose values are determined by intersection graphs can be extended to a 4‑invariant. For the case of $\gl(1|1)$, an explicit 4‑invariant extending it is known. A long‑standing question of S.~Lando reads:

\begin{question}[\cite{ChL07}]\label{conjecture:lando}
    Does there exist a unique 4-invariant that extends the value of the $\sl(2)$-weight system?
\end{question}

A review of this question and related research can be found in~\cite{KL23} and the references therein. We highlight the following results: the unique 4‑invariant that extends the $\sl(2)$-weight system and is defined non‑trivially on all graphs with up to 8 vertices exists~\cite{K21}; some coefficients of the $\sl(2)$-weight system are related to numerical invariants of the corresponding intersection graphs~\cite{CDBook,KLMR14,BNV15}; and there are studies discussing the values of possible extensions on families of graphs~\cite{Z20,Z22,Z24D}. In~\cite{FK25}, the specialization of the $\sl(2)$-weight system corresponding to the 2‑dimensional representation of $\sl(2)$ was extended to arbitrary graphs using a simple deletion-contraction relation; this extension was proved in~\cite{YDJ25} to coincide with the straightforward extension arising from the formula for the standard representation of $\sl(2)$.

Letting $\mathbb{A}$, $\mathbb{G}$ and $\mathbb{I}$ be the complex vector spaces freely spanned by chord diagrams, (non-framed) graphs and intersection graphs respectively and $\Gamma\colon\mathbb{A}\to\mathbb{I}$ be the linearized intersection graph map, Lando's question can be equivalently rephrased in terms of the diagram:
\begin{figure}[H]
    \begin{tikzcd}[column sep=4.5em, row sep=4.2em, ampersand replacement=\&]
        \mathbb{A}
            \arrow[r, "\Gamma"]
            \arrow[d]
            \arrow[dr, "w_{\sl(2)}"{description}] \&
        \mathbb{I}
            \arrow[r, hook]
            \arrow[d, "\tilde{w}_{\sl(2)}"{description}] \&
        \mathbb{G}
            \arrow[d]
            \arrow[dl, dashed, "g"{description}] \\
        \mathbb{A}/\FourT
            \arrow[r, "\hat{w}_{\sl(2)}"'] \&
        \C[c] \&
        \mathbb{G}/\FourT
            \arrow[l, dashed, "\hat{g}"]
    \end{tikzcd}
    \caption{Lando's question: are there unique $g$, $\hat{g}$ making the diagram commute?}
\end{figure}

We answer the question of Lando in the negative by providing such an element (\emph{a certificate}) $\mathcal{C}\in \mathbb{I}$ that $w_{\sl(2)}(\mathcal{C}) \ne 0$ and the image of $\mathcal{C}$ in $\mathbb{G}/\FourT$ is $0$. The calculated value $\tilde{w}_{\sl(2)}(\mathcal{C})$ restricts the set of values of the variable $c$ for which the corresponding evaluation of the $\sl(2)$-weight system can possibly admit extension to $\mathbb{G}/\FourT$ to the roots of $w_{\sl(2)}(\mathcal{C})$. These are $-\frac{3}{32}$, $0$, $\frac{3}{8}$, $1$, and these values are the eigenvalues of the Casimir element in certain representations of $\sl(2)$.

The paper is structured as follows.
In Section~\ref{section:lando} we explain the recipe for obtaining the certificate and specify the four specializations that can possibly extend to 4-invariants.
In Section~\ref{section:3d} we construct an explicit recurrent graph relation for the specialization corresponding to the irreducible 3-dimensional representation of $\sl(2)$.
In Section~\ref{section:halfd} we discuss the specialization conjecturally corresponding to the oscillator representation of $\sl(2)$.
According to our computations, it extends to a unique graph 4-invariant for graphs with $\leq 10$ vertices.
In Section~\ref{section:coefs} we study extensions for the polynomial coefficients of the $\sl(2)$-weight system.
Leading coefficients up to $[c^{n-4}]$ are shown to extend to graph 4-invariants and the others, except $[c^0]$, are proved not to extend.

The question of the uniqueness of extensions of all the specializations and coefficients discussed above remains open. However, our results imply that even when the values of a weight system are completely determined by its intersection graphs, the weight system is not necessarily extendable to a graph 4‑invariant.

\subsection{Weight systems}

\begin{definition}
    For $n\in \N$, a \emph{chord diagram} $D$ of \emph{order} $n$ is a set of pairwise distinct $2n$ points positioned on an oriented circle along with a specific complete pairing of these points. The diagram is considered up to an orientation-preserving diffeomorphism of the circle. We denote the set of all chord diagrams of order $n$ by $\mathbf{A}_n$, and $\mathbf{A} = \bigcup_{n \ge 0} \mathbf{A}_n$. The complex vector space spanned by all the chord diagrams is denoted $\mathbb{A}$.
\end{definition}

In visual representations of chord diagrams, paired points are conventionally connected by a chord, hence the object's name. In this paper, it is assumed that all circles depicted in the figures are oriented counterclockwise.

V.~A.~Vassiliev's theory of finite-type knot invariants~\cite{V90, BN95, CDBook} motivates the study of the space of functions defined on the set of chord diagrams that satisfy specific relations known as the 4T-relations.

\begin{definition}\label{rel:4T}
    Let $M$ be a complex vector space. A linear function $w\colon \mathbb{A} \to M$ is called a \emph{weight system} if the following relations hold:

    \begin{align*}
        w\left(\tchordp{16}{0/8,-4/9}\right)
        -w\left(\tchordp{16}{0/8,-4/7}\right)
        +w\left(\tchordp{16}{0/8,-4/1}\right)
        -w\left(\tchordp{16}{0/8,-4/-1}\right)
        =0.
    \end{align*}
\end{definition}

Henceforth, the following convention is used. The diagrams may include additional chords with endpoints on the dashed arcs. However, only the chords with endpoints on the solid parts of the circles are explicitly depicted. All other chords not visible in the images are consistent across all four chord diagrams.

\begin{remark}
    In the theory of non-framed knot invariants, it is also necessary to impose so-called 1T-relations. When 1T-relations are not assumed, the corresponding functions are referred to as ``framed weight systems''. In this paper, we do not address 1T-relations and simply refer to ``framed weight systems'' as ``weight systems'' for conciseness.
\end{remark}

$\mathbb{A}$ forms a bialgebra with coproduct $D\mapsto \sum_{D'\subseteq D} D' \otimes (D\setminus D')$.
The naturally arising convolution of functions $f, g\colon \mathbb{A} \to R, (f\cdot g)(D) = \sum_{D' \subseteq D} f(D') g(D\setminus D')$, taking value in $\C$-algebra $R$ produces a 4-invariant when $f$ and $g$ are 4-invariants.~\cite{CDBook}

\begin{definition}\label{definition:mult}
    Define the product of two chord diagrams as follows: break the supporting circles of two chord diagrams at a point different from the endpoints, and then glue them together respecting the orientation.

    \begin{figure}[H]
        $\tchord{12}{0/8,-2/6,1/5}
        \tbridgeinner{1}
        \tchord{12}{0/4,2/6}
        =\tchord{10}{1/3,2/4,0/5,-1/-3,-2/-4}
        \sim\tchord{10}{2/4,3/5,0/1,-1/-3,-2/-4}
        =\tchord{12}{0/8,-2/6,1/2}
        \tbridgeinner{2}
        \tchord{12}{0/4,2/6}$
        
        \caption{The product of chord diagrams is unique modulo 4T-relations.}
    \end{figure}

    Different choices of break points, in general, lead to different resulting diagrams. However, the result of the evaluation of a weight system on such a product does not depend on this choice (e.g. see~\cite{CDBook} for details). For chord diagrams $D_1,D_2 \in \mathbf{A}$ and weight system $w$, by $w(D_1 D_2)$ we denote the result of the evaluation of $w$ on any diagram glued from $D_1$ and $D_2$. Similarly we define $w(D_1 D_2)$ for $D_1, D_2 \in \mathbb{A}$.
\end{definition}

\subsection{The \texorpdfstring{$\sl(2)$}{sl(2)}-weight system}
According to~\cite{K93,BN95}, any metrized Lie algebra yields a weight system taking values in the center of its universal enveloping algebra.
In a particular case of this general construction, namely for the Lie algebra $\sl(2)$, S.~Chmutov and A.~Varchenko~\cite{ChV97} proposed a recursive procedure to compute the value of the corresponding weight system without referencing the Lie algebra.
The center of the universal enveloping algebra of $\sl(2)$ is isomorphic to the polynomial algebra $\C[c]$ in one variable, where $c$ is the Casimir element determined by the chosen metric.
Here we take the metric to be the Killing form of $\sl(2)$\foot{Any other choice of the metric on $\sl(2)$ leads to a controllable rescaling of the coefficients of the $\sl(2)$-weight systems. See~\cite{ChV97,CDBook} for the details.}.

\begin{definition}\label{rel:ChV}
	The $\sl(2)$-weight system $w_{\sl(2)}\colon \mathbb{A} \to \C[c]$ is uniquely defined by the following axioms:
    \begin{enumerate}
		\item \emph{Normalization.} The value of $w_{\sl(2)}$ on the chord diagram with one chord equals $c$.
		\item \emph{Multiplicativity.} For two chord diagrams $D_1,D_2 \in \mathbb{A}$, we have $w_{\sl(2)}(D_1 D_2) = w_{\sl(2)}(D_1) w_{\sl(2)}(D_2)$.
		\item \emph{Leaf deletion.} If chord diagram $D$ has a leaf, i.e. a chord that intersects only one other chord, denoting the result of the deletion of the leaf by $D'$, we have $w_{\sl(2)}(D) = \left(c - \frac{1}{2}\right)w_{\sl(2)}(D')$.
        \item \emph{6T-relations.} The following relations hold:
    
        \begin{align*}
            w_{\sl(2)}\left(\tchordp{22}{0/11,-1/4,12/7}\right)
            - &w_{\sl(2)}\left(\tchordp{22}{0/11,1/4,12/7}\right)
            - w_{\sl(2)}\left(\tchordp{22}{0/11,-1/4,10/7}\right)
            + w_{\sl(2)}\left(\tchordp{22}{0/11,1/4,10/7}\right) = \\
            = &\frac{1}{2} w_{\sl(2)}\left(\tchordp{22}{-1/12,4/7}\right)
            - \frac{1}{2} w_{\sl(2)}\left(\tchordp{22}{-1/7,12/4}\right); \\
            w_{\sl(2)}\left(\tchordp{22}{0/11,-1/7,12/4}\right)
            - &w_{\sl(2)}\left(\tchordp{22}{0/11,1/7,12/4}\right)
            - w_{\sl(2)}\left(\tchordp{22}{0/11,-1/7,10/4}\right)
            + w_{\sl(2)}\left(\tchordp{22}{0/11,1/7,10/4}\right) = \\
            = &\frac{1}{2} w_{\sl(2)}\left(\tchordp{22}{-1/12,4/7}\right)
            - \frac{1}{2} w_{\sl(2)}\left(\tchordp{22}{-1/4,12/7}\right). \\
        \end{align*}
    \end{enumerate}
\end{definition}

The above-listed properties are known as \emph{Chmutov-Varchenko relations} for the $\sl(2)$-weight system.
Notably, the 4T-relations are not directly required, but are implied by the defining axioms.

An alternative recurrent way to compute the weight system found in~\cite{ChV97} consists of a single chord deletion relation:
\begin{align}\label{rel:chord-deletion}
    w_{\sl(2)}(D) = \left(c-\frac{k}{2}\right)w_{\sl(2)}(D_a) + \frac{1}{2}\sum_{1\leq i < j\leq k} (w_{\sl(2)}(D_{i,j}'') - w_{\sl(2)}(D_{i,j}^{\times})).
\end{align}
Here chord $a\in D$ is deleted from all right-hand side diagrams, $k$ is the number of $a$'s neighbors (chords intersecting it), the sum goes over the pairs of $a$'s neighbors and $D_{i,j}''$ and $D_{i,j}^{\times}$ correspond to the different ways of rearranging the ends of these chords:
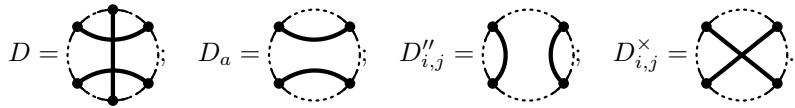
\begin{align*}
    D = \chordp{0.6}{14}{0/7,-2/2,-5/5};\quad
    D_a = \chordp{0.6}{14}{-2/2,-5/5};\quad
    D_{i,j}'' = \chordp{0.6}{14}{-2/-5,2/5};\quad
    D_{i,j}^{\times} = \chordp{0.6}{14}{-2/5,2/-5}.
\end{align*}

By a \emph{specialization} of $w_{\sl(2)}$ at $c = t \in \C$ we mean the evaluation of the resulting polynomial at the specified point. 

\subsection{Graphs and framed graphs}

Every chord diagram $D$ gives rise to the corresponding \emph{intersection graph} $\G_D$ (e.g. see~\cite{chmutovI, chmutovII, chmutovIII} for discussions of applications of this construction to finite type knot invariants theory).

\begin{definition}
    The intersection graph $\G_D$ of a chord diagram $D$ is formed as follows: the set of vertices of $V(\G_D)$ is the set of chords of $D$. Two vertices are connected by an edge if the corresponding chords of the chord diagrams intersect (or, equivalently, the corresponding chords' endpoints interlace).
\end{definition}
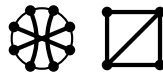
\begin{figure}[H]
    $\tchord{16}{-1/9,1/7,-3/5,-5/3}%
    \quad \begin{tikzpicture}[scale=0.7, tcenter]
        \draw[ultra thick, black, l]
            (0,0) -- (0,1) -- (1,1) -- (1,0) -- (0,0) -- (1,1);
        \draw[black, c=4pt]
            (0,0) node[fill] {}
            (0,1) node[fill] {}
            (1,0) node[fill] {}
            (1,1) node[fill] {};
    \end{tikzpicture}$
    
    \caption{Example of the intersection graph of a chord diagram.}
\end{figure}

\begin{definition}
    A framing of a graph $G$ is a function $f\colon V(G) \to \Z/2\Z$. A graph endowed with a framing is called a framed graph.
\end{definition}

In this paper we only consider simple graphs, whether framed or not. Every graph has the canonical framing, which maps every vertex to $0$. By abuse of notation, we are going to abbreviate the pair $(G,f)$ consisting of a graph and its framing to just $G$. If framing is not specified, it is considered to be canonical.

We follow the convention from~\cite{FK25} for visually depicting graphs.
Framed vertices are displayed as rectangles with their framing inside (by abuse of notation, this framing will sometimes be used as the name of the vertex).
In general, vertices can be depicted explicitly, as labeled vertex groups or be omitted.
Grouped vertices are not shown individually, but are identified by a letter label. Vertex groups can have non-empty intersections, but they never contain explicit vertices.
If explicit vertex $v$ has an edge to group $a$, the graph has edges connecting $v$ to all vertices within $a$. The full set of neighbors of explicit vertex $v$ consists of its explicit vertex neighbors and all vertices in the neighbor groups.
Omitted vertices \& edges and vertex groups are assumed to be consistent between pictures in equations involving multiple graphs.

\begin{figure}[H]
    \centering
    \begin{subfigure}{0.45\textwidth}
        \centering
    	\begin{tikzpicture}[scale=0.6,tcenter,use Hobby shortcut]
    		\draw[ultra thick, black, l]
    			(0,0) -- (1,0)
    			(0,0) -- (0,1)
    			(0,0) -- (1,1)
    			(1,0) -- (0,1)
    			(1,1) -- (0.5,2)
    			(0,1) -- (0.5,2);
    		\draw[black,c=4pt]
    			(0,0) node[fill] {}
    			(1,0) node[fill] {}
    			(0,1) node[fill] {}
    			(1,1) node[fill] {}
    			(0.5,2) node[fill] {};
    		\draw[]
                (0.1,2) node {$w$}
    			(-0.3,-0.3) node {$v$};
    		\draw[thick,l,pattern={Lines[distance=1mm,angle=45,line width=0.2mm]}]
    			([closed]-0.5,1) .. (0.5,0.5) .. (1.5,1) .. (0.5,1.5) .. (-0.5,1);
    		\draw[]
    			(-0.8,1) node {$a$};
    		\draw[thick,l,pattern={Lines[distance=1mm,angle=-45,line width=0.2mm]}]
    			([closed]1,-0.5) .. (0.5,0.5) .. (1,1.5) .. (1.5,0.5) .. (1,-0.5);
    		\draw[]
    			(1.7,0) node {$b$};
    	\end{tikzpicture}%
    	$\leadsto$%
    	\begin{tikzpicture}[scale=0.6,tcenter]
    		\draw[ultra thick, black, l]
    			(1,-0.5) -- (0,0) -- (1,0.5);
    		\draw[black,c=4pt]
    			(0,0) node[fill] {};
    		\draw[]
    			(-0.3,0) node {$v$};
    		\draw[circle]
    			(1,0.5) node[nb] {$a$};
    		\draw[circle]
    			(1,-0.5) node[nb] {$b$};
    	\end{tikzpicture}
        \caption{Non-framed graph.}
    \end{subfigure}
    \begin{subfigure}{0.45\textwidth}
        \centering
    	\begin{tikzpicture}[scale=0.6,tcenter,use Hobby shortcut]
    		\draw[ultra thick, black, l]
    			(0,0) -- (1,0)
    			(0,0) -- (0,1)
    			(0,0) -- (1,1)
    			(1,0) -- (0,1)
    			(1,1) -- (0.5,2)
    			(0,1) -- (0.5,2);
    		\draw[black,cf]
    			(0,0) node[draw,fill=white] {$v$}
    			(0,1) node[draw,fill=white] {$x$}
    			(1,0) node[draw,fill=white] {$y$}
    			(1,1) node[draw,fill=white] {$z$}
    			(0.5,2) node[draw,fill=white] {$w$};
    		\draw[thick,l]
    			([closed]-0.7,1) .. (0.5,0.4) .. (1.7,1) .. (0.5,1.6) .. (-0.7,1);
    		\draw[]
    			(-1,1) node {$a$};
    		\draw[thick,l]
    			([closed]1,-0.7) .. (0.4,0.5) .. (1,1.7) .. (1.6,0.5) .. (1,-0.7);
    		\draw[]
    			(1.9,0) node {$b$};
    	\end{tikzpicture}%
    	$\leadsto$%
    	\begin{tikzpicture}[scale=0.6,tcenter]
    		\draw[ultra thick, black, l]
    			(1,-0.5) -- (0,0) -- (1,0.5);
    		\draw[black,cf]
    			(0,0) node[draw,fill=white] {$v$};
    		\draw[circle]
    			(1,0.5) node[nb] {$a$};
    		\draw[circle]
    			(1,-0.5) node[nb] {$b$};
    	\end{tikzpicture}
        \caption{Framed graph.}
    \end{subfigure}
    \caption{%
        If we only care about the neighborhood of vertex $v$ of a graph provided with a vertex grouping, we abbreviate it as illustrated.
        Here vertex $w$ is outside of the region of interest so we choose to omit it in the abbreviated figure.%
    }
\end{figure}
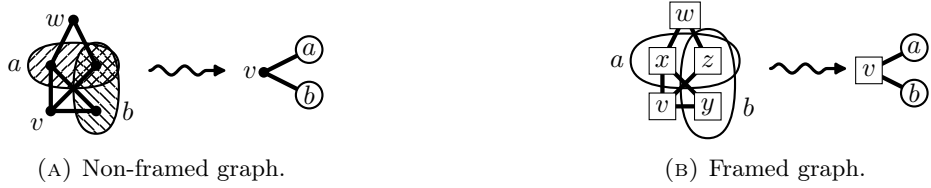

We will make use of adjacency matrices $A(G, \Z/2\Z)$ with coefficients in $\Z/2\Z$ both for non-framed and framed graphs.
For non-framed graphs, the diagonal of the matrix consists of zeroes; for framed graphs, the diagonal entries are the framings of the corresponding vertices.

By $\eta(G)$ we denote the corank of the adjacency matrix $A(G, \Z/2\Z)$.
This function appears in the extension of the weight system associated with the standard representation of $\gl(N)$ to graphs: $w_{\gl(N)}(G) = N^{\eta(G)}$~\cite{BN91,LZ10}.

We denote the sets of isomorphism classes of framed and non-framed graphs by $\mathbf{G}^{fr}$ and $\mathbf{G}$ respectively. The corresponding $\C$-vector spaces spanned by $\mathbf{G}$ and $\mathbf{G}^{fr}$ respectively are denoted $\mathbb{G}$ and $\mathbb{G}^{fr}$.
The following definition takes its origin in~\cite{lando2006j} (see also~\cite{K24,karev2015space}).

\begin{definition}\label{rel:gr4T}
    For complex vector space $M$, linear function $g\colon \mathbb{G}^{fr} \to M$ is called a \emph{4-invariant} if the following graph 4T-relations hold:
    \begin{align*}
        g\left(\begin{tikzpicture}[scale=0.6,tcenter]
            \draw[ultra thick,black,l]
                (0,0) -- (1,0)
                (0,1) -- (1,1)
                (0,0) -- (0,1);
            \draw[black,cf]
                (0,0) node[draw,fill=white] {$b$}
                (0,1) node[draw,fill=white] {$a$};
            \draw[black,circle]
                (1,1) node[nb] {$x$}
                (1,0) node[nb] {$y$};
        \end{tikzpicture}\right)
        -g\left(\begin{tikzpicture}[scale=0.6,tcenter]
            \draw[ultra thick,black,l]
                (0,0) -- (1,0)
                (0,1) -- (1,1);
            \draw[black,cf]
                (0,0) node[draw,fill=white] {$b$}
                (0,1) node[draw,fill=white] {$a$};
            \draw[black,circle]
                (1,1) node[nb] {$x$}
                (1,0) node[nb] {$y$};
        \end{tikzpicture}\right)
        =(-1)^b \left(
        g\left(\begin{tikzpicture}[scale=0.6,tcenter]
            \draw[ultra thick,black,l]
                (0,0) -- (2,0)
                (0,1) -- (2,1)
                (0,0) -- (0,1);
            \draw[black,cf]
                (0,0) node[draw,fill=white] {$b$}
                (0,1) node[draw,fill=white] {$a+b$};
            \draw[black,ellipse]
                (2,1) node[nbe] {$x\triangle y$};
            \draw[black,circle]
                (2,0) node[nb] {$y$};
        \end{tikzpicture}\right)
        -g\left(\begin{tikzpicture}[scale=0.6,tcenter]
            \draw[ultra thick,black,l]
                (0,0) -- (2,0)
                (0,1) -- (2,1);
            \draw[black,cf]
                (0,0) node[draw,fill=white] {$b$}
                (0,1) node[draw,fill=white] {$a+b$};
            \draw[black,ellipse]
                (2,1) node[nbe] {$x\triangle y$};
            \draw[black,circle]
                (2,0) node[nb] {$y$};
        \end{tikzpicture}\right)
        \right).
    \end{align*}
\end{definition}
The non-framed version of 4T-relations is obtained when the framing is canonical.

$\mathbb{G}$ forms a bialgebra with coproduct $G\mapsto \sum_{V'\subseteq V} G_{V'} \otimes G_{V\setminus V'}$.
The naturally arising convolution of functions $f, g\colon \mathbb{G} \to R$ taking value in $\C$-algebra $R$ produces a 4-invariant when $f$ and $g$ are 4-invariants.~\cite{CDBook}

Given a 4-invariant $g\colon \mathbb{G} \to M$, we can define a weight system $w$ by the rule: for $D \in \mathbf{A}$, $w(D) := g(\G_D)$. Notice that this construction works since the application of the intersection graph map to diagrams involved in a 4T-relation produces a graph 4T-relation.

Sometimes we will focus on stronger relations that imply 4T:
\begin{definition}
    For framed graph $G$, its 2T transform for pair $(a, b)$ of different vertices is expressed as follows:
    \begin{align*}
        \begin{tikzpicture}[tcenter,scale=0.8]
            \draw[ultra thick,black,l]
                (1,1) -- (0,1) (0,0) -- (1,0);
            \draw[ultra thick,black,dotted,l]
                (0,1) -- (0,0);
            \draw[black,cf]
                (0,0) node[draw,fill=white] {$b$}
                (0,1) node[draw,fill=white] {$a$};
            \draw[black,circle]
                (1,1) node[nb] {$x$}
                (1,0) node[nb] {$y$};
        \end{tikzpicture}
        \mapsto
        \begin{tikzpicture}[tcenter,scale=0.8]
            \draw[ultra thick,black,l]
                (2,1) -- (0,1) (0,0) -- (2,0);
            \draw[ultra thick,black,dotted,l]
                (0,1) -- (0,0);
            \draw[black,cf]
                (0,0) node[draw,fill=white] {$a+b$}
                (0,1) node[draw,fill=white] {$a$};
            \draw[black,circle]
                (2,1) node[nb] {$x$};
            \draw[black,ellipse]
                (2,0) node[nbe] {$x\triangle y$};
        \end{tikzpicture},
    \end{align*}
    where the dashed edge (absent or present initially) is preserved when $a=0$ and switched when $a=1$.
\end{definition}

The 2T-relations then are equality of the value of a function on graph $G$ and its value on its 2T transform.

In terms of adjacency matrices, the 2T transform for $(a, b)$ corresponds to adding $a$'s row to $b$'s row and then $a$'s column to $b$'s column. This operation corresponds to a basis change and does not alter the rank of the matrix.

Given a weight system, we can try to extend to a graph 4-invariant. Even though we know nothing about the uniqueness of the extension to non-framed graphs, an extension of a weight system to the space of framed graphs, if it exists, is not unique (see~\cite{karev2015space} for the discussion).

This paper deals with weight systems and 4-invariants with values in $\C$ and $\C[t]$ only.  For the details see, e.g.,~\cite{CDBook}. We refer to~\cite{KL23} as a state-of-the-art review of weight systems and 4-invariants.

\subsection{Accompanying materials}
\begin{enumerate}
    \item Certificate of non-extendability of $w_{\sl(2)}$: \url{https://github.com/fomichev-d/sl2-cert}.
    \item Table for $c=-\frac{3}{32}$: \url{https://github.com/fomichev-d/sl2_at_-3_32}.
    \item Linear regression pipeline for the $[c^{n-3}]$ coefficient: \url{https://github.com/Serdabel/w_sl2.a_3}.
\end{enumerate}

\subsection{Acknowledgements}
The authors thank S.~Chmutov for valuable comments on the coefficients of the $\sl(2)$-weight system, P.~Zakorko for pointing out 4-invariance of the natural extension of the weight system corresponding to the three-dimensional representation of $\sl(2)\cong\so(3)$ to graphs, and M.~Kazarian for an independent verification of non-extendability of the $\sl(2)$-weight system.

Support from the Basic Research Program of HSE University is gratefully acknowledged (HSE-BR-2025-023).

\section{\texorpdfstring{$\sl(2)$}{sl(2)}-weight system does not extend to a graph 4-invariant}\label{section:lando}

E.~Krasilnikov computed the unique $\sl(2)$-weight system extension to a 4-invariant of graphs with $\leq 8$ vertices~\cite{K21}.
For graphs on $n$ vertices, the original approach consisted of computing $w_{\sl(2)}$ on intersection graphs using Chmutov--Varchenko relations, iteratively extending the table to graphs for which there is a 4T-relation involving the given graph and three graphs with already computed values, solving the remaining system, which we refer to as the final system, as a linear system of 4T-relations with known values substituted and finally verifying that all 4T-relations hold.

\begin{table}[H]
    \centering
    \begin{tabular}{r|ccccccccccc}
        $n$ & 0 & 1 & 2 & 3 & 4 & 5 & 6 & 7 & 8 & 9 & 10 \\
        \hline
        $\dim \mathbb{G}_n$ & 1 & 1 & 2 & 4 & 11 & 34 & 156 & 1044 & 12346 & 274668 & 12005168 \\
        $\dim \mathbb{I}_n$ & 1 & 1 & 2 & 4 & 11 & 34 & 154 & 978 & 9497 & 127954 & 2165291
    \end{tabular}
    \caption{The number of graphs and intersection graphs on $n$ vertices.}
\end{table}

Using a similar approach and making use of optimizations including parallel computations and row-sparsity of the final system, we attempted to extend the system further.

In more detail, for each $n_i$ our algorithm does the following:
\begin{enumerate}
	\item Iterates chord diagrams with $n_i$ chords, computes and deduplicates their intersection graphs in order to iterate all intersection graphs on $n$ vertices.
		For each intersection graph, it computes the value of the $w_{\sl(2)}$ using the chord deletion Chmutov--Varchenko relation~(\ref{rel:chord-deletion}) and saves it to the table of computed values.

		Here and elsewhere table insertion and queries use a graph hash to reduce the required number of graph isomorphism checks.
		Krasilnikov's algorithm's hash was the multiset of vertex degrees.
		Our hash is finer, consisting of two lists: one with numbers of vertices of degrees $0, \ldots, n_i-1$ (providing the same information as the initial hash) and the other with numbers of vertices of 2-degrees $0,\ldots,n_i$, where vertex 2-degree is defined as the number of vertices to which there exists a path of degree exactly 2.
	\item Iteratively attempts to extend the table with 4T-relations.
		For every graph with unknown value, which are iterated in parallel, all 4T-relations with this graph are considered as a system of linear equalities.
        Each row corresponding to a 4T-relation consists of coefficients in $\Q$ at indices corresponding to graphs with unknown values and the right-hand side value in $\Q$ corresponding to substituted values of graphs with previously known values; at most four left-hand side coefficients are non-zero and are $\pm 1$.
		If this system admits a solution for the graph in question, it is used to extend the table at the end of the iteration.
		Once a full iteration produces no additions, the iterative step is finished.

		This iterative step is more expressive than Krasilnikov's since it considers the entire system of 4T-relations involving a given graph and not each of those relations separately.
	\item Constructs and solves a linear system of all 4T-relations with the remaining graphs.
		Each 4T-relation has at most four non-zero terms so a sparse representation for matrix rows is used.
		Gaussian elimination is eagerly done as new rows are added; if value of any graph is obtained, the table is extended and the row is removed.
		If a zero row with a non-zero right-hand side is obtained at any point, the algorithm terminates since the system has no solutions.
	\item Continues if no graphs remain and terminates otherwise since the extension is not unique.
	\item After all $n_i = 0, \ldots, n$ have been processed successfully, all 4T-relations are checked in parallel.
\end{enumerate}

Our program reproduced E.~Krasilnikov's table of values of $w_{\sl(2)}$ for $n \leq 8$ and failed to extend $w_{\sl(2)}$ to $n=9$ as the final system had no solutions.
Subsequently we obtained a certificate $\mathcal{C}\in \mathbb{I}$ proving that the $\sl(2)$-weight system indeed does not extend to graphs with $9$ vertices.

The certificate we computed is a linear combination of 3300 graph 4T-relations expanding to a linear combination of 5006 intersection graphs evaluating to non-zero polynomial $c\left(c-\frac{3}{8}\right)(c-1)(c+\frac{3}{32})$.
Thus we obtain:
\begin{proposition}\label{proposition:lando}\leavevmode
    \begin{enumerate}
        \item The $\sl(2)$-weight system does not extend to a 4-invariant of graphs.
            In other words, \textbf{Lando's question has a negative answer, and even in the case when a weight system factors through the intersection graph map, it is not necessarily extendable to a graph 4-invariant}.
        \item It is only possible for specializations of the $\sl(2)$-weight system to extend to 4-invariants of graphs for $c=0,\frac{3}{8},1,-\frac{3}{32}$.
    \end{enumerate}
\end{proposition}
\begin{proof}
    The certificate can be found at \url{https://github.com/fomichev-d/sl2-cert}.
\end{proof}

The factors of the certificate value correspond to eigenvalues of the Casimir element evaluated on specific representations of $\sl(2)$. As Casimir is a central element of the universal enveloping algebra of $\sl(2)$, it acts as a scalar operator in any irreducible representation. For $G\in \mathbf{G}$, let $V$ denote the set of vertices of $G$. Extension of the $w_{\sl(2)}$ may be possible only for those Casimir eigenvalues:
\begin{itemize}
    \item $c=0$ corresponds to the value of $w_{\sl(2)}$ for the 1-dimensional representation of $\sl(2)$. In this case we have the trivial extension $0^{|V|}$;
    \item $c=\frac{3}{8}$ corresponds to the irreducible 2-dimensional representation of $\sl(2)$ and admits an extension studied in~\cite{FK25, YDJ25}:
        \begin{align*}
        	\phi(G) = 2^{-2|V|} \sum_{U\subseteq V} \left(-\frac{1}{2}\right)^{|V\setminus U|} 2^{\eta(G|_U)};
        \end{align*}
    \item $c=1$ corresponds to the irreducible 3-dimensional representation of $\sl(2, \C) \cong \so(3, \C)$ and admits an extension described and investigated in Section~\ref{section:3d};
    \item $c=-\frac{3}{32}$ is conjectured to correspond to the oscillator representation of $\sl(2)$ and to admit an extension.  The oscillator representation and its decomposition into Verma modules are discussed in Section~\ref{section:halfd}.
    Our computations show that this extension exists and is unique for graphs with $\leq 10$ vertices.
\end{itemize}

\section{Extension of \texorpdfstring{$\sl(2)$}{sl(2)}-weight system specialized to the irreducible 3-dimensional representation}\label{section:3d}
\subsection{Specialization of \texorpdfstring{$\sl(2)$}{sl(2)}-weight system for the chord diagrams}

The Lie algebras $\sl(2,\C)$ and $\so(3,\C)$ are isomorphic. The standard way to see it is to consider the regular representation of $\sl(2,\C)$. The representation map $\sl(2,\C) \to \gl(\sl(2,\C))$ is injective, and the representation operators respect the Killing form, which is non-degenerate. Since the dimensions match, isomorphism follows.

The irreducible 3-dimensional representation of $\sl(2, \C)$ is the defining representation of $\so(3, \C)$. The standard Casimir element in this representation has a single eigenvalue $c = 1$.

\begin{theorem}\cite{BN91, BN95}\label{theorem:soN-repr}
    Consider all mappings $\sigma$ from chords of diagram $D$ to $\setp{1,-1}$.
    For $\sigma$, define $|\sigma|$ as the number of connected components left if all chords are doubled and each chord with $\sigma(c)=-1$ is then twisted:
    
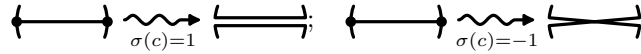
\begin{align*}
		\begin{tikzpicture}[tcenter,scale=0.6]
			\draw[black,very thick,l]
				(160:1cm) arc (160:200:1cm)
				(20:1cm) arc (20:-20:1cm)
				(-1,0) -- (1,0);
			\draw[black,c=4pt]
				(-1,0) node[fill] {}
				(1,0) node[fill] {};
		\end{tikzpicture}%
		\leadsto_{\sigma(c)=1}%
		\begin{tikzpicture}[tcenter,scale=0.6]
			\draw[black,very thick,l]
				(160:1cm) arc (160:175:1cm) -- (5:1cm) arc (5:20:1cm)
				(200:1cm) arc (200:185:1cm) -- (-5:1cm) arc (-5:-20:1cm);
		\end{tikzpicture};
        \quad
		\begin{tikzpicture}[tcenter,scale=0.6]
			\draw[black,very thick,l]
				(160:1cm) arc (160:200:1cm)
				(20:1cm) arc (20:-20:1cm)
				(-1,0) -- (1,0);
			\draw[black,c=4pt]
				(-1,0) node[fill] {}
				(1,0) node[fill] {};
		\end{tikzpicture}%
		\leadsto_{\sigma(c)=-1}%
		\begin{tikzpicture}[tcenter,scale=0.6]
			\draw[black,very thick,l]
				(160:1cm) arc (160:175:1cm) -- (-5:1cm) arc (-5:-20:1cm)
				(200:1cm) arc (200:185:1cm) -- (5:1cm) arc (5:20:1cm);
		\end{tikzpicture}.
    \end{align*}

    Then the following formula holds:
    \begin{align*}
        w_{\sl(2)}(D)|_{c=1}= 2^{-|D|} \sum_{\sigma\in\setp{1,-1}^{|D|}} \left(\prod_{c\in D} \sigma(c)\right) 3^{|\sigma| - 1}.
    \end{align*}
\end{theorem}

\subsection{Extension to graphs}
Recall that for framed graph $G$, $\eta(G)$ is the corank of the adjacency matrix $A(G, \Z/2\Z)$.

\begin{proposition}\cite{LZ10,CL72}\label{proposition:eta-extends}
    $\eta(G)$ is a 2T-invariant that extends the value of $|\sigma|-1$ for diagrams $D$, where $\sigma(c)=1$ for each $c\in D$.
\end{proposition}

\begin{definition}\label{definition:psi}
    Graph function $\psi\colon\mathbb{G}\to\C$ is defined by:
    \begin{align*}
        \psi(G) := 2^{-|V|} \sum_{V'\subseteq V} (-1)^{|V'|} 3^{\eta(G')},
    \end{align*}
    where $G'\in\mathbf{G}^{fr}$ differs from $G\in\mathbf{G}$ in assigning framing $1$ to vertices $V'$ and $0$ to $V\setminus V'$.
\end{definition}

\begin{proposition}\label{proposition:psi-4T}
    $\psi$ is a 4-invariant.
\end{proposition}
\begin{proof}
    Consider a non-framed 4T relation:
    
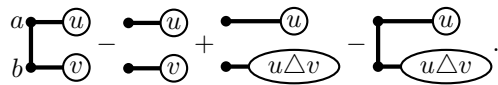
\begin{align*}
        \begin{tikzpicture}[tcenter,scale=0.6]
			\draw[ultra thick,black,l]
				(1,1) -- (0,1) (0,0) -- (1,0)
				(0,1) -- (0,0);
			\draw[black,c=4pt]
				(0,0) node[fill] {}
				(0,1) node[fill] {};
			\draw[black,circle]
				(1,1) node[nb] {$u$}
				(1,0) node[nb] {$v$}
                (-0.3,1) node {$a$}
                (-0.3,0) node {$b$};
		\end{tikzpicture}
        -\begin{tikzpicture}[tcenter,scale=0.6]
			\draw[ultra thick,black,l]
				(1,1) -- (0,1) (0,0) -- (1,0);
			\draw[black,c=4pt]
				(0,0) node[fill] {}
				(0,1) node[fill] {};
			\draw[black,circle]
				(1,1) node[nb] {$u$}
				(1,0) node[nb] {$v$};
		\end{tikzpicture}
		+\begin{tikzpicture}[tcenter,scale=0.6]
			\draw[ultra thick,black,l]
				(1.5,1) -- (0,1) (0,0) -- (1.5,0);
			\draw[black,c=4pt]
				(0,0) node[fill] {}
				(0,1) node[fill] {};
			\draw[black,circle]
				(1.5,1) node[nb] {$u$};
			\draw[black,ellipse]
				(1.5,0) node[nbe=34pt] {$u\triangle v$};
		\end{tikzpicture}
		-\begin{tikzpicture}[tcenter,scale=0.6]
			\draw[ultra thick,black,l]
				(1.5,1) -- (0,1) (0,0) -- (1.5,0)
				(0,1) -- (0,0);
			\draw[black,c=4pt]
				(0,0) node[fill] {}
				(0,1) node[fill] {};
			\draw[black,circle]
				(1.5,1) node[nb] {$u$};
			\draw[black,ellipse]
				(1.5,0) node[nbe=34pt] {$u\triangle v$};
		\end{tikzpicture}.
    \end{align*}
    
    Consider the inner sum obtained while varying framings on $a$, $b$ with the rest of $V'$ fixed.
    With $(-1)^{|V'|}$ accounted for, this sum splits into four framed 4T relations on $3^{\eta(G')}$.
    Since $\eta(G')$ is a 2T-invariant, it is a framed 4T-invariant, and so is $3^{\eta(G')}$.
    The sum evaluates to zero.
\end{proof}

The statements listed above imply:

\begin{proposition}
    $\psi$ is a 4-invariant extending $w_{\sl(2)}$ at $c=1$.
\end{proposition}


\subsection{A recurrent relation for \texorpdfstring{$\psi$}{ψ}}
\begin{theorem}\label{theorem:del-vert-3d}
    Let $G$ be a graph, $x$ its vertex, $\deg x = k \geq 1$, $A = \setp{a_1, \ldots, a_k}$ its neighbors.
    Then the following recurrent relation takes place:
    \begin{align*}
        \psi(G)
        =\frac{(-1)^{k+1}}{2} \psi(\bar{G}_x)
        +\frac{(-1)^k}{4} \psi(\bar{G}_x^{a_1})
        +\frac{1}{4} \psi(G_x^{a_1}),
    \end{align*}
    where $\bar{G}_x$ is obtained from $G$ by deleting $x$ and inverting all edges within $G|_A$,
    $\bar{G}_x^{a_1}$ is obtained from $\bar{G}_x$ by applying the 2T transform to all pairs $(a_1, a_i)$ for $i=2, \ldots, k$ and then deleting $a_1$,
    $G_x^{a_1}$ is obtained from $G$ by deleting $x$, applying the same 2T transforms and deleting $a_1$.
\end{theorem}
\begin{example}
    \begin{align*}
        \psi\left(\begin{tikzpicture}[tcenter, scale=0.6]
            \draw[ultra thick,black,l]
                (0,0) -- (0.7,1) -- (2,1)
                (0,0) -- (1,0) -- (2,0)
                (0,0) -- (0.7,-1) -- (2,-1)
                (0.7,1) -- (1,0);
			\draw[black,c=4pt]
				(0,0) node[fill] {}
				(0.7,1) node[fill] {}
				(1,0) node[fill] {}
				(0.7,-1) node[fill] {};
			\draw[black,circle]
				(2,1) node[nb] {$a$}
				(2,0) node[nb] {$b$}
				(2,-1) node[nb] {$c$};
        \end{tikzpicture}\right)
        =\frac{1}{2} \psi\left(\begin{tikzpicture}[tcenter, scale=0.6]
            \draw[ultra thick,black,l]
                (0.7,1) -- (2,1)
                (1,0) -- (2,0)
                (0.7,-1) -- (2,-1)
                (0.7,1) -- (0.7,-1) -- (1,0);
			\draw[black,c=4pt]
				(0.7,1) node[fill] {}
				(1,0) node[fill] {}
				(0.7,-1) node[fill] {};
			\draw[black,circle]
				(2,1) node[nb] {$a$}
				(2,0) node[nb] {$b$}
				(2,-1) node[nb] {$c$};
        \end{tikzpicture}\right)
        -\frac{1}{4} \psi\left(\begin{tikzpicture}[tcenter, scale=0.6]
            \draw[ultra thick,black,l]
                (1,0) -- (2.5,0)
                (0.7,-1) -- (2.5,-1)
                (0.7,-1) -- (1,0);
			\draw[black,c=4pt]
				(1,0) node[fill] {}
				(0.7,-1) node[fill] {};
			\draw[black,ellipse]
				(2.5,0) node[nbe] {$a\triangle b$}
				(2.5,-1) node[nbe] {$a\triangle c$};
        \end{tikzpicture}\right)
        +\frac{1}{4} \psi\left(\begin{tikzpicture}[tcenter, scale=0.6]
            \draw[ultra thick,black,l]
                (1,0) -- (2.5,0)
                (0.7,-1) -- (2.5,-1);
			\draw[black,c=4pt]
				(1,0) node[fill] {}
				(0.7,-1) node[fill] {};
			\draw[black,ellipse]
				(2.5,0) node[nbe] {$a\triangle b$}
				(2.5,-1) node[nbe] {$a\triangle c$};
        \end{tikzpicture}\right).
    \end{align*}
\end{example}
\begin{proof}
    Consider the part of the adjacency matrix of $G$ corresponding to vertices $x$, $a_1, \ldots, a_k$ in this order:
    \begin{align*}
        \begin{pmatrix}
            0 & 1 & 1 & \cdots & 1 \\
            1 & 0 & e_{1,2} & \cdots & e_{1,k} \\
            1 & e_{1,2} & 0 & \cdots & e_{2,k} \\
            \vdots & \vdots & \vdots & \ddots & \vdots \\
            1 & e_{1,k} & e_{2,k} & \cdots & 0 \\
        \end{pmatrix}.
    \end{align*}

    By definition of $\psi(G)$, we are interested in a sum with terms corresponding to all possible framings on the diagonal:
    \begin{align*}
        \begin{pmatrix}
            x & 1 & 1 & \cdots & 1 \\
            1 & a_1 & e_{1,2} & \cdots & e_{1,k} \\
            1 & e_{1,2} & a_2 & \cdots & e_{2,k} \\
            \vdots & \vdots & \vdots & \ddots & \vdots \\
            1 & e_{1,k} & e_{2,k} & \cdots & a_k \\
        \end{pmatrix}.
    \end{align*}

    Now we will perform transforms $\eta(G)$ is invariant against and account for the coefficients in the end.
    First we split the sum into two parts, one with $x=0$ and the other with $x=1$.

    Consider the $x=1$ case:
    \begin{align*}
        \begin{pmatrix}
            1 & 1 & 1 & \cdots & 1 \\
            1 & a_1 & e_{1,2} & \cdots & e_{1,k} \\
            1 & e_{1,2} & a_2 & \cdots & e_{2,k} \\
            \vdots & \vdots & \vdots & \ddots & \vdots \\
            1 & e_{1,k} & e_{2,k} & \cdots & a_k \\
        \end{pmatrix}.
    \end{align*}

    Since $\eta$ is 2T-invariant, we apply 2T transforms for pairs $(x, a_i)$ for $i=1,\ldots,k$:
    \begin{align*}
        \begin{pmatrix}
            1 & 0 & 0 & \cdots & 0 \\
            0 & \bar{a_1} & \bar{e_{1,2}} & \cdots & \bar{e_{1,k}} \\
            0 & \bar{e_{1,2}} & \bar{a_2} & \cdots & \bar{e_{2,k}} \\
            \vdots & \vdots & \vdots & \ddots & \vdots \\
            0 & \bar{e_{1,k}} & \bar{e_{2,k}} & \cdots & \bar{a_k} \\
        \end{pmatrix}.
    \end{align*}
    
    At this point the matrix decomposes into a $1\times 1$ identity block matrix for $x$ and the rest of the matrix.
    Removing the vertex (and its associated row and column) does not change the corank:
    \begin{align*}
        \begin{pmatrix}
            \bar{a_1} & \bar{e_{1,2}} & \cdots & \bar{e_{1,k}} \\
            \bar{e_{1,2}} & \bar{a_2} & \cdots & \bar{e_{2,k}} \\
            \vdots & \vdots & \ddots & \vdots \\
            \bar{e_{1,k}} & \bar{e_{2,k}} & \cdots & \bar{a_k} \\
        \end{pmatrix}.
    \end{align*}

    This corresponds to $\bar{G}_x$, the first term of the right-hand side.
    Now we can account for the coefficient $\frac{(-1)^{k+1}}{2}$: $\frac{1}{2}$ comes from removing vertex $x$, $(-1)$ from $x$ having framing $1$, $(-1)^k$ from $a_1, \ldots, a_k$ having their framings reversed.

    Consider the $x=0$ case:
    \begin{align*}
        \begin{pmatrix}
            0 & 1 & 1 & \cdots & 1 \\
            1 & a_1 & e_{1,2} & \cdots & e_{1,k} \\
            1 & e_{1,2} & a_2 & \cdots & e_{2,k} \\
            \vdots & \vdots & \vdots & \ddots & \vdots \\
            1 & e_{1,k} & e_{2,k} & \cdots & a_k \\
        \end{pmatrix}.
    \end{align*}

    Here we apply 2T transforms for pairs $(a_1, a_i)$ for $i=2,\ldots,k$:
    \begin{align*}
        \begin{pmatrix}
            0 & 1 & 0 & \cdots & 0 \\
            1 & a_1 & a_1+e_{1,2} & \cdots & a_1+e_{1,k} \\
            0 & a_1+e_{1,2} & a_1+a_2 & \cdots & a_1+e_{1,2}+e_{1,k}+e_{2,k} \\
            \vdots & \vdots & \vdots & \ddots & \vdots \\
            0 & a_1+e_{1,k} & a_1+e_{1,2}+e_{1,k}+e_{2,k} & \cdots & a_1+a_k \\
        \end{pmatrix}.
    \end{align*}

    In particular, in place of $e_{i,j}$ this matrix will have $a_1+e_{1,i}+e_{1,j}+e_{i,j}$ for $i,j>1$.

    Since the first row contains a single $1$, we can kill the whole column under it; similarly with the first column:
    \begin{align*}
        \begin{pmatrix}
            0 & 1 & 0 & \cdots & 0 \\
            1 & 0 & 0 & \cdots & 0 \\
            0 & 0 & a_1+a_2 & \cdots & a_1+e_{1,2}+e_{1,k}+e_{2,k} \\
            \vdots & \vdots & \vdots & \ddots & \vdots \\
            0 & 0 & a_1+e_{1,2}+e_{1,k}+e_{2,k} & \cdots & a_1+a_k \\
        \end{pmatrix}.
    \end{align*}

    The matrix decomposes into blocks once again so $x$ and $a_1$ can be removed:
    \begin{align*}
        \begin{pmatrix}
            a_1+a_2 & \cdots & a_1+e_{1,2}+e_{1,k}+e_{2,k} \\
            \vdots & \ddots & \vdots \\
            a_1+e_{1,2}+e_{1,k}+e_{2,k} & \cdots & a_1+a_k \\
        \end{pmatrix}.
    \end{align*}

    We split the sum again, this time for $a_1$.

    Consider the subcase of $a_1=0$:
    \begin{align*}
        \begin{pmatrix}
            a_2 & \cdots & e_{1,2}+e_{1,k}+e_{2,k} \\
            \vdots & \ddots & \vdots \\
            e_{1,2}+e_{1,k}+e_{2,k} & \cdots & a_k \\
        \end{pmatrix}.
    \end{align*}

    This corresponds to $G_x^{a_1}$, the third term of the right-hand side.
    The coefficient is $\frac{1}{4}$, which comes from removing vertices $x$ and $a_1$.

    Consider the subcase of $a_1=1$:
    \begin{align*}
        \begin{pmatrix}
            \bar{a_2} & \cdots & \bar{e_{1,2}+e_{1,k}+e_{2,k}} \\
            \vdots & \ddots & \vdots \\
            \bar{e_{1,2}+e_{1,k}+e_{2,k}} & \cdots & \bar{a_k} \\
        \end{pmatrix}.
    \end{align*}

    This corresponds to $\bar{G}_x^{a_1}$, the second term of the right-hand side.
    The coefficient is $\frac{(-1)^k}{4}$: $\frac{1}{4}$ comes from removing vertices $x$ and $a_1$, $(-1)$ from $a_1=1$, $(-1)^{k-1}$ from $a_2,\ldots,a_k$ having their framing reversed.
\end{proof}

A similar relation might also hold for the extension $\phi$ of the weight system corresponding to the 2-dimensional representation of $\sl(2)$.
It is:
\begin{conjecture}\label{conj:del-vert-2d}
    \begin{align*}
        \phi(G) = \left(\frac{3}{8}-1\right)\phi(G_x) + \frac{(-1)^{k+1}}{2}\phi(\bar{G}_x) + \frac{(-1)^k}{4}\cdot\frac{3}{8}\phi(\bar{G}_x^{a_1}) + \frac{1}{4}\cdot\frac{3}{8}\phi(G_x^{a_1}).
    \end{align*}
\end{conjecture}
Similarly to previous conventions, $G_x$ is $G\setminus\setp{x}$.
The conjecture was verified numerically for all graphs with $\leq 9$ vertices.
Below we prove it for the case of $k=2$ neighbors.

One might be tempted to replace $\frac{3}{8}$ with $c$ and consider $w_{\sl(2)}$.
Before we attempt to do so, we prove a lemma for $w_{\sl(2)}$ as a weight system on chord diagrams.

\begin{lemma}
    Let $D$ be a chord diagram with a chosen chord $x$, where \emph{$x$ intersects exactly two other chords}, which we will refer to as $a$ and $b$.
    Then the following equalities hold:
    \begin{align*}
        w_{\sl(2)}(\tchordp{20}{0/10,-3/3,13/7})
        &=(c-1) w_{\sl(2)}(\tchordp{20}{-3/3,13/7})
        -\frac{1}{2} w_{\sl(2)}(\tchordp{20}{-3/7,13/3})
        +\frac{1}{2} w_{\sl(2)}(\tchordp{20}{-3/13,7/3}); \\
        w_{\sl(2)}(\tchordp{20}{0/10,-3/7,13/3})
        &=(c-1) w_{\sl(2)}(\tchordp{20}{-3/7,13/3})
        -\frac{1}{2} w_{\sl(2)}(\tchordp{20}{-3/3,13/7})
        +\frac{1}{2} w_{\sl(2)}(\tchordp{20}{-3/13,3/7}).
    \end{align*}

    Moreover, if the diagram is such that to one side from $x$ there are no other endpoints between the endpoints of $x$ and the corresponding endpoints of $a$ and $b$, this equality can be expressed as follows:
    \begin{align}\label{rel:chord-k2}
        w_{\sl(2)}(\tchordp{20}{0/10,-1/3,11/7})
        &=(c-1) w_{\sl(2)}(\tchordp{20}{-1/3,11/7})
        -\frac{1}{2} w_{\sl(2)}(\tchordp{20}{-1/7,11/3})
        +\frac{1}{2} c \cdot w_{\sl(2)}(\tchordp{20}{7/3}); \\
        \nonumber
        w_{\sl(2)}(\tchordp{20}{0/10,-1/7,11/3})
        &=(c-1) w_{\sl(2)}(\tchordp{20}{-1/7,11/3})
        -\frac{1}{2} w_{\sl(2)}(\tchordp{20}{-1/3,11/7})
        +\frac{1}{2} c \cdot w_{\sl(2)}(\tchordp{20}{3/7}).
    \end{align}
\end{lemma}
\begin{proof}
    The main relations are a trivial specialization of the Chmutov--Varchenko chord deletion relation~(\ref{rel:chord-deletion}) for the case of two neighbors.

    Relations (\ref{rel:chord-k2}) are obtained from the fact that one of the two chords in the last diagram in both relations is isolated, thus it can be removed with an addition of the $c$ multiplier by the usual Chmutov--Varchenko relations.
\end{proof}

Relations (\ref{rel:chord-k2}) can be written in terms of intersection graphs as follows:
\begin{align}\label{rel:graph-k2}
    w_{\sl(2)}(G) = (c-1) w_{\sl(2)}(G_x) - \frac{1}{2} w_{\sl(2)}(\bar{G}_x) + \frac{1}{2} c \cdot w_{\sl(2)}(G_x^a).
\end{align}
Note that the lemma proves it specifically for those intersection graphs that have a chord diagram with two adjacent pairs of endpoints.

For the specific case of the 2-dimensional representation $\phi$ at $c=\frac{3}{8}$, the lemma holds \emph{for all intersection graphs} because $\phi$ satisfies graph Chmutov--Varchenko relations; thus we obtain (\ref{rel:graph-k2}) for $\phi$ for arbitrary graphs.

All available data (extension of $w_{\sl(2)}$ to graphs with $\leq 8$ vertices, explicitly constructed $c=1$ specialization per Theorem~\ref{theorem:del-vert-3d}, the table of values for the $c=-3/32$ specialization for graphs with $\leq 10$ vertices) suggests that equality (\ref{rel:graph-k2}) might hold for all intersection graphs.
However, its generalization to the case of $k$ neighbors in the form similar to Conjecture~\ref{conj:del-vert-2d} fails on graphs with $\leq 8$ vertices.

\section{The oscillator representation}\label{section:halfd}
\subsection{The oscillator representation of \texorpdfstring{$\sl(2)$}{sl(2)}}
Consider the following operators on the space of entire functions on $\C^n$ (the Fock model uses polynomials $\C[z_1,\ldots,z_n]$):
\begin{align*}
    E=-\frac{1}{2}\sum \frac{\partial^2}{\partial z_i^2}; \quad
    F=\frac{1}{2}\sum z_i^2; \quad
    H=-\frac{n}{2}-\sum z_i \frac{\partial}{\partial z_i}.
\end{align*}
These operators satisfy all Lie algebra relations for $\sl(2)$ and generate \emph{the oscillator representation} of $\sl(2)$ on $\C^n$, or \emph{the Segal--Shale--Weil representation} (in a highest-weight reformulation --- we swap $E$ and $F$ and negate $H$).~\cite{G17}

We focus on the specific case of $n=1$:
\begin{align*}
    E=-\frac{1}{2}\frac{\partial^2}{\partial z^2}; \quad
    F=\frac{z^2}{2}; \quad
    H=-\frac{1}{2}-z\frac{\partial}{\partial z}.
\end{align*}
The Casimir element $C=\frac{1}{4}(EF+FE+\frac{1}{2}HH)$ here is simply the scalar operator of multiplication by $-\frac{3}{32}$.

Alternatively, this corresponds to the sum of infinite-dimensional Verma modules $M(-\frac{1}{2})\oplus M(-\frac{3}{2})$, where $M(\lambda)$ for $\lambda\in\C$ is generated by $v_i$, $i\geq 0$, $v_{-1}=0$, with operators:
\begin{align*}
    E v_i = (\lambda - i + 1) v_{i-1}; \quad
    F v_i = (i + 1) v_{i+1}; \quad
    H v_i = (\lambda - 2i) v_i.
\end{align*}
The Casimir element here is $C=\frac{1}{8}\lambda^2 + \frac{1}{4}\lambda$.~\cite{H08}
It evaluates to $-\frac{3}{32}$ for both $\lambda=-\frac{1}{2}$ and $\lambda=-\frac{3}{2}$. The highest weight vectors that generate the corresponding Verma modules as submodules of the oscillator representation are $1$ and $z$.

Given that the last remaining specialization of the $\sl(2)$-weight system possibly admitting extension is $c=-\frac{3}{32}$, we conjecture that such an extension exists and is related to the oscillator representation of $\sl(2)$.
\begin{conjecture}
    The specialization of the $\sl(2)$-weight system at $c=-\frac{3}{32}$ extends to a 4-invariant of graphs $\xi(G)$.
\end{conjecture}
The extension has been constructed for graphs with $\leq 10$ vertices and it has been verified that it is unique for these graphs.
The table of computed values can be found at \url{https://github.com/fomichev-d/sl2_at_-3_32}.

\subsection{A conjectured recurrent relation on graphs}
\begin{conjecture}
    Suppose that $\xi(G)$ is a 4-invariant of graphs extending the specialization of the $\sl(2)$-weight system at $c=-\frac{3}{32}$.
    Let $G$ be a graph, $x$ its vertex, $\deg x = 2$, $A = \setp{a_1, a_2}$ its neighbors.
    Then the following recurrent relation takes place:
    \begin{align*}
        \xi(G) = \left(\frac{-3}{32}-1\right)\xi(G_x) + \frac{-1}{2}\xi(\bar{G}_x) + \frac{1}{2}\cdot\frac{-3}{32}\xi(G_x^{a_1}),
    \end{align*}
    where $G_x$ is obtained from $G$ by deleting $x$,
    $\bar{G}_x$ is obtained from $G$ by deleting $x$ and inverting the edge $(a_1, a_2)$,
    $G_x^{a_1}$ is obtained from $G_x$ by applying the 2T transform to the pair $(a_1, a_2)$ and deleting $a_1$.
\end{conjecture}
This relation has been verified for all graphs with $\leq 10$ vertices.

\section{Polynomial coefficients of the \texorpdfstring{$\sl(2)$}{sl(2)}-weight system}\label{section:coefs}
Since $w_{\sl(2)}$ is a polynomial of degree $n$, another natural question to ask is: which coefficients of this polynomial extend to 4-invariants of graphs?
It is known that:
\begin{itemize}
    \item the leading coefficient $[c^n]$ is $1$;
    \item the second coefficient $[c^{n-1}]$ extends as $-\frac{|E|}{2}$~\cite{ChV97};
    \item the third coefficient $[c^{n-2}]$ explicitly extends to graphs~\cite{ChV97}. An explicit formula can be found in Subsection~\ref{subsection:n-3};
    \item the free coefficient $[c^0]$ is $0$ for $n>0$ and $1$ for $n=0$ (this extension corresponds to evaluation on the 1-dimensional representation of $\sl(2)$ discussed earlier).
\end{itemize}

But first consider the $\sl(2)$-weight system as defined on chord diagrams.
Theorem~3 in~\cite{ChV97} implies that coefficient $[c^{n-k}]$ of $w_{\sl(2)}$ on chord diagrams with $n \geq 2k$ chords as a weight system is a convolution $\mathbf{1}_{n-2k}\cdot f_{2k}$, where $f_{2k}$ is a weight system defined on diagrams with $2k$ chords and $\mathbf{1}_{l}$ stands for the weight system that takes value $1$ on any chord diagram with $l$ chords and $0$ otherwise.

We will now prove a similar theorem for graphs:
\begin{theorem}\label{thm:convolution}
    If coefficient $[c^{n-k}]$ of $w_{\sl(2)}$ admits extension $f_{2k}$ to a 4-invariant of graphs on $\leq 2k$ vertices, it admits extension $f$ to a 4-invariant of arbitrary graphs, where:
    \begin{align*}
        f(G) = \begin{cases}
            f_{2k}(G), & |V| \leq 2k; \\
            (\mathbf{1}_{|V|-2k} \cdot f_{2k})(G), & |V| \geq 2k.
        \end{cases}
    \end{align*}
\end{theorem}
\begin{proof}
    We need to prove that $f = \mathbf{1}_{|V|-2k} \cdot f_{2k}$ is a 4-invariant extending $[c^{n-k}]$.
    
    It is a 4-invariant: for each $n$, it is a convolution of two 4-invariants.

    It extends $[c^{n-k}]$. Consider intersection graph $G$:
    \begin{align*}
        f(G)
        =(\mathbf{1}_{|V|-2k}\cdot f_{2k})(G)
        =\sum_{\substack{V' \subseteq V \\ |V'| = 2k}} f_{2k}(G|_{V'})
        =\sum_{\substack{V' \subseteq V \\ |V'| = 2k}} [c^{n-k}]w_{\sl(2)}(G|_{V'})
        =[c^{n-k}] w_{\sl(2)}(G).
    \end{align*}
    
    We use the fact that each $G|_{V'}$ itself is an intersection graph; the last equality follows from Theorem~3 in~\cite{ChV97}.
\end{proof}

\begin{corollary}
    Higher coefficients $[c^{n-k}]$ of $w_{\sl(2)}$ for $k \leq 4$ admit extension to 4-invariants of graphs.
    Higher coefficients $[c^{n-k}]$ for $k\geq 5$ do not admit such extensions.
    Lower coefficients $[c^k]$ for $k \geq 1$ do not extend.
\end{corollary}
\begin{proof}
    Coefficients for $k \leq 4$ extend because $w_{\sl(2)}$ extends to graphs with $\leq 8$ vertices.

    Coefficients $k=5,6,7,8$ do not extend, which is shown by the certificate in Proposition~\ref{proposition:lando}: there is a combination of 4T-relations expanding to a combination of intersection graphs on $9$ vertices evaluating to the polynomial $c\left(c + \frac 3{32}\right)\left(c - \frac 38\right)\left(c - 1\right) = c^4-\frac{41}{32}c^3+\frac{63}{256}c^2+\frac{9}{256}c$.

    To show why coefficients $k>8$ do not extend either, we introduce the following transform of the certificate.
    By Chmutov-Varchenko relations, adding a leaf to an intersection graph multiplies the original $w_{\sl(2)}$ value by $c-\frac{1}{2}$.
    Adding an isolated vertex multiplies the value by $c$.
    Thus a transform mapping intersection graph $G$ to the difference of $G$ with an isolated vertex added and $G$ with an arbitrary leaf added multiplies the value by $\frac{1}{2}$ (this operation was discussed in detail in~\cite{K24}).
    If we apply this transform to the original certificate $m$ times, we will have a polynomial with the same non-zero coefficients obtained as a combination of 4T-relations expanding to a combination of intersection graphs on $9+m$ vertices.
    This time these coefficients will correspond to $k=5+m,6+m,7+m,8+m$.

    Coefficients $[c^k]$ for $k \geq 1$ do not extend because adding $k-1$ isolated vertices to each graph of the original certificate produces a counterexample for graphs on $8+k$ vertices.
\end{proof}

\subsection{The case of \texorpdfstring{$n-3$}{n-3}}\label{subsection:n-3}

It was shown above that the leading coefficients of the $\sl(2)$-weight system extend to graphs, closed formulae for the first ones are already known.
In this section we develop a general tool that reconstructs such closed formulae from finite data and apply it to obtain coefficient $[c^{n-3}]$.

To motivate the approach, let us first look at the already known formulae written in the same graphical notation that we will use for the main formula below:
\begin{itemize}
	\item $[c^n] = \emptyset$;
	\item $[c^{n-1}] = -\frac{1}{2}\,\cgraph{0.4}{2}{0/1}$;
	\item $[c^{n-2}] = \frac{1}{4}\,\cgraph{0.4}{4}{0/1,2/3} + \frac{1}{4}\,\cgraph{0.4}{3}{-1/0,0/1} - \frac{1}{4}\,\cgraph{0.4}{3}{0/1,1/2,2/0} + \frac{1}{2}\,\cgraph{0.4}{4}{0/1,1/2,2/3,3/0} - \cgraph{0.4}{4}{0/1,0/2,0/3,1/2,1/3,2/3}$%
    \foot{%
        On chord diagrams the same coefficient is classically written as $\frac{1}{4}\frac{e(e-1)}{2} - \frac{1}{4}t + \frac{1}{2}q_{\mathrm{cd}}$, where $e,t$ count crossings and triangles and $q_{\mathrm{cd}}$ counts \emph{quadrangles}~\cite{ChV97}.
        The discrepancy with the expression above arises because quadrangles are defined differently for chord diagrams and graphs: in language of intersection graphs, a quadrangle is a four-vertex \emph{induced} subgraph containing a spanning $4$-cycle, of which there are three types $\left(\cgraph{0.4}{4}{0/1,1/2,2/3,3/0}, \cgraph{0.4}{4}{0/1,0/2,0/3,1/2,3/2}, \cgraph{0.4}{4}{0/1,0/2,0/3,1/2,1/3,2/3}\right)$, whereas our expression counts $\cgraph{0.4}{4}{0/1,1/2,2/3,3/0}$ and $\cgraph{0.4}{4}{0/1,0/2,0/3,1/2,1/3,2/3}$ range over all, not necessarily induced, subgraphs.
        A single quadrangle of a chord diagram is therefore recorded once by $q_{\mathrm{cd}}$ but contributes through several graph subgraphs; matching the two conventions gives $q_{\mathrm{cd}} = \cgraph{0.4}{4}{0/1,1/2,2/3,3/0} - 2\,\cgraph{0.4}{4}{0/1,0/2,0/3,1/2,1/3,2/3}$, which turns the chord-diagram formula into the one above and accounts for the extra term $-\,\cgraph{0.4}{4}{0/1,0/2,0/3,1/2,1/3,2/3}$.%
    }.
\end{itemize}

Here each small graph $G_i$ stands for the graph-counting function $f_i$ of the depicted graph: this is the same as in the feature matrix (Definition~\ref{def:feature-matrix} below), so $\cgraph{0.4}{3}{0/1,1/2,2/0}$, for instance, denotes $f_i$ with $G_i$ the triangle.
In particular, the term $\frac{e(e-1)}{8}$, which, multiplied by $4$, counts unordered pairs of edges, is itself of this form: two edges either are disjoint or share a vertex, so it splits as $\frac{1}{4}\,\cgraph{0.4}{4}{0/1,2/3} + \frac{1}{4}\,\cgraph{0.4}{3}{-1/0,0/1}$, which is why it already appears in such form above.

Each of these coefficients is a closed-form expression~--- a linear combination of only \emph{finitely many} subgraph-counting functions.
This is a non-trivial property: decomposing in an arbitrary basis (e.g. the dual basis $\setp{G_i^*}$) generally gives an infinite linear combination.
This suggests that the higher coefficients could be of this form as well.

Since $w_{\sl(2)}$ takes value in polynomials and $[c^k]$ and $[c^{n-k}]$ are linear functions on degree $n$ polynomials, each coefficient of the polynomial is a weight system by itself.
In particular, extended $[c^{n-3}] w_{\sl(2)}$ is a function on arbitrary graphs and may be viewed as a linear function on the space spanned by graphs like any other graph function.
Viewing it this way, we may state our task as that of decomposing $[c^{n-3}] w_{\sl(2)}$ in the basis of subgraph-counting functions $f_i$.

\subsubsection{Linear regression}

To recover such a linear decomposition, we use the standard statistical method of \emph{linear regression}, commonly used in machine learning.
It represents each object by a real feature vector $x$ and predicts a target value as $\hat{y} = w^\top x = w_1 x_1 + \ldots + w_d x_d$, where $w$ is a vector of real weights fixed for a given model.
Given
\begin{itemize}
    \item a real feature matrix $X \in \R^{l \times d}$, whose rows correspond to objects (each object being a vector in $\R^d$) and whose columns correspond to the features computed on those objects, and
    \item a target vector $y \in \R^l$ holding the values to be predicted,
\end{itemize}
linear regression seeks the weight vector $w$ minimizing the squared Euclidean distance between $y$ and the predictions $Xw$:
\begin{align*}
    \min_w \norm{Xw - y}_2^2.
\end{align*}
This minimization is equivalent to the normal equation $X^\top X w = X^\top y$, and whenever $X$ has full column rank its solution exists and is unique:
\begin{align*}
    w = (X^\top X)^{-1} X^\top y.
\end{align*}

In our setting the objects are graphs on at most $n$ vertices, the target is the value of the sought coefficient computed on each of them, the features are the subgraph-counting functions; this finite feature matrix $F_{\leq n}$ is constructed in the next subsection.
Crucially, $F_{\leq n}$ is square and unitriangular, hence invertible, so the regression solution collapses to an \emph{exact} change of coordinates rather than a statistical fit:
\begin{align*}
    w = (F_{\leq n}^\top F_{\leq n})^{-1} F_{\leq n}^\top y
    = F_{\leq n}^{-1} (F_{\leq n}^\top)^{-1} F_{\leq n}^\top y
    = F_{\leq n}^{-1} y.
\end{align*}
\emph{Although linear regression is normally an approximate, overdetermined procedure, in our case it returns the exact and unique decomposition of the coefficient.}

According to Theorem~\ref{thm:convolution}, for the coefficient $[c^{n-3}]$ it suffices to consider graphs on at most $2 \cdot 3 = 6$ vertices.
$\dim{\mathbf{G}_{\leq 6}}$ is already 208, which makes computation by hand infeasible and justifies computational treatment.

\subsubsection{The feature matrix}

Let $\mathbf{G}_n$ be the set of graphs on $n$ vertices and $\mathbf{G}_{\leq n}$ be the set of graphs on $\leq n$ vertices, so that $\mathbf{G} = \bigsqcup_{i=1}^{\infty} \mathbf{G}_i$ and $\mathbf{G}_{\leq n} = \bigsqcup_{i=1}^{n} \mathbf{G}_i$. 

We fix an ordering of all graphs obeying the following rules:
\begin{enumerate}
    \item $v(G_i) < v(G_j) \Longrightarrow i < j$, i.e. among two graphs the one with the higher number of vertices has a larger index;
    \item $v(G_i) = v(G_j)$ and $e(G_i) < e(G_j) \Longrightarrow i < j$, i.e. for two graphs on the same number of vertices the one with the higher number of edges has a larger index;
    \item on graphs with the same numbers of vertices and edges any order may be chosen; we fix one and use it throughout.
\end{enumerate}
The third point leaves a degree of freedom; to avoid ambiguity, we use the ordering of the Graph Atlas~\cite{atlas} since this is the ordering of the library used in the computational part.
Once the ordering is fixed, we identify graphs with their indices and treat $\mathbf{G}_n$, $\mathbf{G}_{\leq n}$ as the corresponding sets of indices.

We now define the central object.
\begin{definition}\label{def:feature-matrix}
    Let $F$ be the infinite matrix whose entry $F_{i,j}$ equals the number of subgraphs of $G_i$ isomorphic to $G_j$.
    The row and column indices are those of the ordering above.
    Following the common machine-learning convention (rows are objects, columns are features), we call $F$ the \emph{feature matrix}.
\end{definition}

We treat $F$ as a formal object and carry out all linear algebra with its finite principal submatrices.
Matrix $F$ is lower unitriangular: by the chosen ordering a graph $G_j$ with $j > i$ can never be a subgraph of $G_i$, while the number of subgraphs of $G_i$ inside $G_i$ itself equals one, giving $1$ on the diagonal.

For index subsets $R, C, J$ we abbreviate submatrices of $F$ as $F_{R,C}$ (rows indexed by $R$, columns by $C$) and $F_J := F_{J,J}$.
In particular, the square matrix $F_{\mathbf{G}_{\leq n}}$ is again lower unitriangular.
\begin{figure}[H]
    \begin{tabular}{C|CCCCCCCC}
        & \emptyset & \cgraph{0.3}{1}{} & \cgraph{0.3}{2}{} & \cgraph{0.3}{2}{0/1} & \cgraph{0.3}{3}{} & \cgraph{0.3}{3}{1/2} & \cgraph{0.3}{3}{1/2,2/0} & \cgraph{0.3}{3}{0/1,1/2,2/0} \\
        \hline
        \emptyset & 1 \\
        \cgraph{0.3}{1}{} & 1 & 1 \\
        \cgraph{0.3}{2}{} & 1 & 2 & 1 \\
        \cgraph{0.3}{2}{0/1} & 1 & 2 & 1 & 1 \\
        \cgraph{0.3}{3}{} & 1 & 3 & 3 & 0 & 1 \\
        \cgraph{0.3}{3}{1/2} & 1 & 3 & 3 & 1 & 1 & 1 \\
        \cgraph{0.3}{3}{1/2,2/0} & 1 & 3 & 3 & 2 & 1 & 2 & 1 \\
        \cgraph{0.3}{3}{0/1,1/2,2/0} & 1 & 3 & 3 & 3 & 1 & 3 & 3 & 1
    \end{tabular}
    \caption{$F_{\mathbf{G}_{\leq 3}}$.}
\end{figure}

Write $f_j(G_i) := F_{i,j}$ for the $j$-th feature of the graph $G_i$, i.e. the number of subgraphs of $G_i$ isomorphic to $G_j$.
For each $n$, let $\mathbb{G}_{\leq n}$ be the vector space spanned by $\mathbf{G}_{\leq n}$ with basis $e = \setpp{G_j}{G_j \in \mathbf{G}_{\leq n}}$ and $\mathbb{G}_{\leq n}^*$ its dual space of linear functions.
Each feature $f_j$ with $G_j \in \mathbf{G}_{\leq n}$ restricts to an element of $\mathbb{G}_{\leq n}^*$.

The following lemma records the change of basis at each finite degree.

\begin{lemma}[Change of basis at degree $n$]\label{lem:change-of-basis}
    On $\mathbf{G}_{\leq n}$ features $\setpp{f_j}{G_j \in \mathbf{G}_{\leq n}}$ form a basis of the dual space, the coordinates $b_{\mathbf{G}_{\leq n}} = (b_j)_{G_j \in \mathbf{G}_{\leq n}}$ of arbitrary function $\tau|_{\mathbf{G}_{\leq n}}$ in this basis are
    \begin{align*}
        b_{\mathbf{G}_{\leq n}} = F_{\mathbf{G}_{\leq n}}^{-1} \tau(\mathbf{G}_{\leq n}),
        \qquad
        \tau|_{\mathbf{G}_{\leq n}} = \sum_{G_j \in \mathbf{G}_{\leq n}} b_j f_j .
    \end{align*}
    In particular, this decomposition exists and is unique.
\end{lemma}

\begin{proof}
    The values $\tau(\mathbf{G}_{\leq n}) = (\tau(G_j))_{G_j \in \mathbf{G}_{\leq n}}$ are the coordinates of $\tau|_{\mathbf{G}_{\leq n}}$ in the dual basis $e^* = (G_j^*)_{G_j \in \mathbf{G}_{\leq n}}$.
    The columns of $F_{\mathbf{G}_{\leq n}}$ express the features $f_j$ in the basis $e^*$, so $F_{\mathbf{G}_{\leq n}}$ is the transition
    matrix from $e^*$ to $f$.
    Being lower unitriangular, it is invertible; hence the features $\setpp{f_j}{G_j \in \mathbf{G}_{\leq n}}$ form a basis and the coordinates in it are obtained by a single left multiplication by $F_{\mathbf{G}_{\leq n}}^{-1}$.
\end{proof}
	
\subsubsection{Reconstruction from finite restrictions}

The coordinates produced by Lemma~\ref{lem:change-of-basis} do not depend on the degree $n$ at which they are computed: increasing it only adds new coordinates and leaves existing ones untouched.

\begin{lemma}[Consistency of degrees $n$ and $n+1$]\label{lem:consistency}
    For every function $\tau$ the entries of $b_{\mathbf{G}_{\leq n+1}}$ indexed by $\mathbf{G}_{\leq n}$ coincide with $b_{\mathbf{G}_{\leq n}}$.
\end{lemma}

\begin{proof}
    Consider the computation above carried out at degree $n+1$. Since the ordering places all graphs on $\leq n$ vertices before those on $n+1$ vertices, splitting $\mathbf{G}_{\leq n+1}$ into $\mathbf{G}_{\leq n}$ and $\mathbf{G}_{n+1}$ puts $F_{\mathbf{G}_{\leq n+1}}$ in the block form
    \begin{align*}
        F_{\mathbf{G}_{\leq n+1}} =
        \begin{pmatrix}
            F_{\mathbf{G}_{\leq n}} & 0 \\[3pt]
            B & F_{\mathbf{G}_{n+1}}
        \end{pmatrix},
        \qquad B := F_{\mathbf{G}_{n+1}, \mathbf{G}_{\leq n}}.
    \end{align*}
    The upper-right block is zero because a graph on $\leq n$ vertices has no subgraph on $n+1$ vertices, and both diagonal blocks are lower unitriangular.
    For such a block-triangular matrix the inverse has the same block structure with the diagonal blocks inverted in place:
    \begin{align*}
        F_{\mathbf{G}_{\leq n+1}}^{-1} =
        \begin{pmatrix}
            F_{\mathbf{G}_{\leq n}}^{-1} & 0 \\[3pt]
            -F_{\mathbf{G}_{n+1}}^{-1} B F_{\mathbf{G}_{\leq n}}^{-1} & F_{\mathbf{G}_{n+1}}^{-1}
        \end{pmatrix};
    \end{align*}
    in particular, its upper-left corner is exactly $F_{\mathbf{G}_{\leq n}}^{-1}$.
    As the upper-right block vanishes, the entries of $b_{\mathbf{G}_{\leq n+1}} = F_{\mathbf{G}_{\leq n+1}}^{-1} \tau(\mathbf{G}_{\leq n+1})$ indexed by $\mathbf{G}_{\leq n}$ depend only on $\tau(\mathbf{G}_{\leq n})$ and equal $F_{\mathbf{G}_{\leq n}}^{-1} \tau(\mathbf{G}_{\leq n}) = b_{\mathbf{G}_{\leq n}}$.
\end{proof}

\begin{corollary}\label{cor:reconstruction}
    The coefficients of a function $\tau$ in the feature basis $f$ are well-defined and can be reconstructed from finite degrees.
    They can be computed iteratively as
    \begin{align*}
        b_{\mathbf{G}_{\leq n}} = F_{\mathbf{G}_{\leq n}}^{-1} \tau(\mathbf{G}_{\leq n}),
    \end{align*}
    and the values obtained at degree $n$ are final: they are never altered at any larger degree.
\end{corollary}

Now we can obtain a formula in a form similar to~\cite{F80}.

\begin{theorem}\label{thm:cn3-formula}
	Applying Corollary~\ref{cor:reconstruction} with $\tau = [c^{n-3}]$ over $\mathbf{G}_{\leq 6}$ yields the following closed formula:
	\begin{align*}
		[c^{n-3}] = {}& -\frac{1}{8}\tcgraph{4}{0/1,0/2,0/3} - \frac{1}{8}\tcgraph{4}{0/1,1/2,2/3} + \frac{1}{8}\tcgraph{4}{0/1,1/2,2/3,3/1} - \frac{1}{8}\tcgraph{4}{0/1,0/2,0/3,1/2,3/2} + \frac{3}{8}\tcgraph{4}{0/1,0/2,0/3,1/2,1/3,2/3} - \frac{1}{8}\tcgraph{5}{-1/0,0/1,2/3} \\
		& + \frac{1}{8}\tcgraph{5}{-1/0,0/1,1/-1,2/3} - \frac{1}{4}\tcgraph{5}{0/1,1/2,2/3,3/4,4/1} - \frac{3}{8}\tcgraph{5}{0/1,1/2,2/3,3/4,4/0} + \frac{1}{4}\tcgraph{5}{-1/0,0/1,1/-1,1/2,2/3,3/4} + \frac{1}{4}\tcgraph{5}{-1/0,0/1,1/2,1/3,-1/-2,-1/-3} + \frac{1}{2}\tcgraph{5}{0/1,1/2,1/3,1/4,2/3,2/4,3/4} \\
		& - \frac{1}{2}\tcgraph{5}{0/1,0/2,0/4,1/2,1/4,2/3,2/4,3/4} - \frac{1}{2}\tcgraph{5}{0/1,0/-1,0/2,0/-2,-1/1,1/2,-1/-2,-2/2} + \frac{3}{2}\tcgraph{5}{0/1,0/2,0/3,0/4,1/2,1/3,1/4,2/4,3/4} - \frac{7}{2}\tcgraph{5}{0/1,0/2,0/3,0/4,1/2,1/3,1/4,2/3,2/4,3/4} - \frac{1}{8}\tcgraph{6}{0/1,2/5,3/4} - \frac{1}{4}\tcgraph{6}{0/1,2/3,3/4,4/5,5/2} \\
		& + \frac{1}{4}\tcgraph{6}{0/1,1/2,2/3,3/4,4/5,5/0} + \frac{1}{2}\tcgraph{6}{0/1,2/3,2/4,2/5,3/4,3/5,4/5} - \frac{1}{2}\tcgraph{6}{0/-1,0/3,0/1,1/2,-1/-2,2/3,-2/-3} - \frac{1}{2}\tcgraph{6}{0/1,0/2,0/5,1/2,1/3,2/3,3/4,4/5} + \tcgraph{6}{0/-1,0/1,1/2,-1/-2,1/-2,-1/2,2/3,-2/-3} + \tcgraph{6}{0/1,0/2,0/3,0/5,1/2,1/3,2/3,3/4,4/5} + \tcgraph{6}{0/1,0/4,0/5,1/2,2/3,2/5,3/4,3/5,4/5} \\
		& + \frac{1}{2}\tcgraph{6}{0/-1,0/3,0/1,-1/1,1/2,-1/-2,2/3,-2/-3,2/-2} - 3\tcgraph{6}{0/-1,0/3,0/1,1/2,1/4,2/3,2/5,3/4,4/5} - 2\tcgraph{6}{0/1,0/2,0/4,0/5,1/2,1/4,1/5,2/3,3/4,4/5} - 2\tcgraph{6}{0/1,0/4,0/5,1/2,1/5,2/3,2/5,3/4,3/5,4/5} - \tcgraph{6}{0/1,0/2,0/5,1/2,1/3,1/5,2/3,2/4,3/4,4/5} + \tcgraph{6}{0/1,0/4,0/5,1/2,1/3,1/5,2/3,2/4,3/4,4/5} + 4\tcgraph{6}{0/1,0/2,0/3,0/5,1/2,1/3,1/5,2/3,2/5,3/4,4/5} \\
		& - 2\tcgraph{6}{0/1,0/4,0/5,1/2,1/3,1/4,1/5,2/3,2/4,3/4,4/5} + 2\tcgraph{6}{0/1,0/4,0/5,1/2,1/4,1/5,2/3,2/4,3/4,3/5,4/5} + 4\tcgraph{6}{0/1,0/3,0/-1,-1/1,1/2,-1/-2,1/-2,-1/2,-2/2,2/3,-2/-3} - 4\tcgraph{6}{0/1,0/2,0/4,0/5,1/2,1/3,1/4,1/5,2/3,3/4,3/5,4/5} + 2\tcgraph{6}{0/1,0/-1,0/2,0/-2,1/2,1/3,1/5,2/3,2/4,3/4,3/5,4/5} + 4\tcgraph{6}{0/1,0/-1,0/2,0/-2,0/3,1/2,1/3,1/5,2/3,2/4,3/4,3/5,4/5} - 8\tcgraph{6}{0/1,0/-1,0/2,0/-2,0/3,1/2,1/3,1/4,1/5,2/3,2/4,2/5,3/4,3/5,4/5}.
	\end{align*}
\end{theorem}

\begin{proof}
    The data and the linear-regression pipeline used to obtain this formula are available at \url{https://github.com/Serdabel/w_sl2.a_3}.
\end{proof}

\printbibliography

\end{document}